\documentclass[oneside,english,reqno]{amsart}
\usepackage[T1]{fontenc}
\usepackage[latin9]{inputenc}
\usepackage{geometry}
\usepackage{babel}
\usepackage{refstyle}
\usepackage{float}
\usepackage{mathrsfs}
\usepackage{amstext}
\usepackage{amsthm}
\usepackage{amssymb}
\usepackage{graphicx}
\usepackage{esint}
\PassOptionsToPackage{normalem}{ulem}
\usepackage{ulem}
\usepackage[unicode=true,pdfusetitle,
 bookmarks=true,bookmarksnumbered=false,bookmarksopen=false,
 breaklinks=false,pdfborder={0 0 1},backref=false,colorlinks=false]
 {hyperref}
\hypersetup{
 colorlinks=true,citecolor=blue,linkcolor=blue,linktocpage=true}

\makeatletter

\AtBeginDocument{\providecommand\secref[1]{\ref{sec:#1}}}
\AtBeginDocument{\providecommand\defref[1]{\ref{def:#1}}}
\AtBeginDocument{\providecommand\exaref[1]{\ref{exa:#1}}}
\AtBeginDocument{\providecommand\corref[1]{\ref{cor:#1}}}
\AtBeginDocument{\providecommand\remref[1]{\ref{rem:#1}}}
\AtBeginDocument{\providecommand\figref[1]{\ref{fig:#1}}}
\AtBeginDocument{\providecommand\thmref[1]{\ref{thm:#1}}}
\RS@ifundefined{subref}
  {\def\RSsubtxt{section~}\newref{sub}{name = \RSsubtxt}}
  {}
\RS@ifundefined{thmref}
  {\def\RSthmtxt{theorem~}\newref{thm}{name = \RSthmtxt}}
  {}
\RS@ifundefined{lemref}
  {\def\RSlemtxt{lemma~}\newref{lem}{name = \RSlemtxt}}
  {}

\numberwithin{equation}{section}
\numberwithin{figure}{section}
\numberwithin{table}{section}
\usepackage{enumitem}		
\theoremstyle{plain}
\newtheorem{thm}{\protect\theoremname}[section]
  \theoremstyle{definition}
  \newtheorem{defn}[thm]{\protect\definitionname}
  \theoremstyle{remark}
  \newtheorem{note}[thm]{\protect\notename}
  \theoremstyle{plain}
  \newtheorem{lem}[thm]{\protect\lemmaname}
  \theoremstyle{plain}
  \newtheorem{cor}[thm]{\protect\corollaryname}
  \theoremstyle{remark}
  \newtheorem{rem}[thm]{\protect\remarkname}
  \theoremstyle{definition}
  \newtheorem{example}[thm]{\protect\examplename}
  \theoremstyle{plain}
  \newtheorem*{assumption*}{\protect\assumptionname}
  \theoremstyle{remark}
  \newtheorem*{acknowledgement*}{\protect\acknowledgementname}

\providecommand{\MR}[1]{}

\usepackage{needspace}
\usepackage{bbm}

\newref{lem}{refcmd={Lemma \ref{#1}}}
\newref{thm}{refcmd={Theorem \ref{#1}}}
\newref{cor}{refcmd={Corollary \ref{#1}}}
\newref{sec}{refcmd={Section \ref{#1}}}
\newref{sub}{refcmd={Section \ref{#1}}}
\newref{chap}{refcmd={Chapter \ref{#1}}}
\newref{prop}{refcmd={Proposition \ref{#1}}}
\newref{exa}{refcmd={Example \ref{#1}}}
\newref{tab}{refcmd={Table \ref{#1}}}
\newref{rem}{refcmd={Remark \ref{#1}}}
\newref{def}{refcmd={Definition \ref{#1}}}
\newref{fig}{refcmd={Figure \ref{#1}}}
\newref{claim}{refcmd={Claim \ref{#1}}}

\@ifundefined{showcaptionsetup}{}{%
 \PassOptionsToPackage{caption=false}{subfig}}
\usepackage{subfig}
\AtBeginDocument{
  
}

\makeatother

  \providecommand{\acknowledgementname}{Acknowledgement}
  \providecommand{\assumptionname}{Assumption}
  \providecommand{\corollaryname}{Corollary}
  \providecommand{\definitionname}{Definition}
  \providecommand{\examplename}{Example}
  \providecommand{\lemmaname}{Lemma}
  \providecommand{\notename}{Note}
  \providecommand{\remarkname}{Remark}
\providecommand{\theoremname}{Theorem}

\begin{document}

\title{Transfer Operators, Induced Probability Spaces, and Random Walk Models}

\author{Palle Jorgensen and Feng Tian}

\address{(Palle E.T. Jorgensen) Department of Mathematics, The University
of Iowa, Iowa City, IA 52242-1419, U.S.A. }

\email{palle-jorgensen@uiowa.edu}

\urladdr{http://www.math.uiowa.edu/\textasciitilde{}jorgen/}

\address{(Feng Tian) Department of Mathematics, Hampton University, Hampton,
VA 23668, U.S.A.}

\email{feng.tian@hamptonu.edu}

\subjclass[2000]{Primary 47L60, 46N30, 65R10, 58J65, 81S25.}

\keywords{Unbounded operator, closable operator, spectral theory, discrete
analysis, distribution of point-masses, probability space, stochastic
processes, discrete time, path-space measure, endomorphism, harmonic
functions.}

\maketitle
\pagestyle{myheadings}
\markright{}
\begin{abstract}
We study a family of discrete-time random-walk models. The starting
point is a fixed generalized transfer operator $R$ subject to a set
of axioms, and a given endomorphism in a compact Hausdorff space $X$.
Our setup includes a host of models from applied dynamical systems,
and it leads to general path-space probability realizations of the
initial transfer operator. The analytic data in our construction is
a pair $\left(h,\lambda\right)$, where $h$ is an $R$-harmonic function
on $X$, and $\lambda$ is a given positive measure on $X$ subject
to a certain invariance condition defined from $R$. With this we
show that there are then discrete-time random-walk realizations in
explicit path-space models; each associated to a probability measures
$\mathbb{P}$ on path-space, in such a way that the initial data allows
for spectral characterization: The initial endomorphism in $X$ lifts
to an automorphism in path-space with the probability measure $\mathbb{P}$
quasi-invariant with respect to a shift automorphism. The latter takes
the form of explicit multi-resolutions in $L^{2}$ of $\mathbb{P}$
in the sense of Lax-Phillips scattering theory.
\end{abstract}

\tableofcontents{}

\section{Introduction}

We study a family of stochastic processes indexed by a discrete time
index. Our results encompass the more traditional random walk models,
but our study here goes beyond that. The processes considered are
generated by a single positive operator, say $R$, defined on $C\left(X\right)$
where $X$ is a given compact Hausdorff space. From a given positive
operator $R$ we then derive an associated system of generalized transition
probabilities, and an induced probability space; the induction realized
as a probability space of infinite paths having $X$ as a base space.
In order for us to pin down the probability space, i.e., the induced
path-space measure $\mathbb{P}$, two more ingredients will be needed,
one is a prescribed endomorphism $\sigma$ in $X$, consistent with
$R$; and, the other, is a generalized harmonic function $h$ on $X$,
i.e., $R\left(h\right)=h$. We further explore the interplay between
the harmonic functions $h$ and the associated path-space measures
$\mathbb{P}$. To do this we note that the given endomorphism $\sigma$
in $X$ induces an automorphism in the path-space. We show that the
path-space measure $\mathbb{P}$ is quasi-invariant, and we compute
the corresponding Radon-Nikodym derivative. Our motivation derives
from the need to realize multiresolution models in in a general setting
of dynamical systems as they arise in a host of applications: in symbolic
dynamics, e.g., \cite{MR2030387,MR1898809}, in generalized multiresolution
model, e.g., \cite{MR2520021}; in dynamics arising from an iteration
of substitutions, e.g., \cite{MR1128089}; in geometric measure theory,
and for Iterated Function Systems (IFS), e.g., \cite{MR625600,MR2529881};
or in stochastic analysis, e.g., \cite{MR2966130,MR625600,MR3411557,MR3411549,MR3392913,MR3392911,MR3360371,MR3221687,MR3014680}.

\section{\label{sec:setting}The Setting}

Organization of the paper: The setup starts with a fixed and given
compact Hausdorff space $X$, and positive operator $R$ (a generalized
transfer operator), defined on the function algebra $C\left(X\right)$,
and subject to two simple axioms. Candidates for $X$ will include
compact Bratteli diagrams, see e.g., \cite{MR2205435,MR2041268,MR1812067,MR1194074}.
Given $R$, in principle one is then able to derive a system of generalized
transition probabilities for discrete time processes starting from
points in $X$; see details in the present section. However, in order
to build path-space probability spaces this way, more considerations
are required, and this will be explored in detail in Sections \ref{sec:L1}
and \ref{sec:EA} below. \secref{IFS} deals with a subfamily of systems
where the generalized transfer operator $R$ is associated to an Iterated
Function system (IFS). In \secref{Harm}, we derive some conclusions
from the main theorems in the paper.

Let $X$ be a compact Hausdorff space, and $\mathcal{M}\left(X\right)$
be the space of all measurable functions on $X$. Let $R:C\left(X\right)\longrightarrow\mathcal{M}\left(X\right)$
be a positive linear mapping, i.e., $f\geq0$ $\Longrightarrow$ $Rf\geq0$. 
\begin{defn}
\label{def:R}Let $\mathscr{L}\left(R\right)$ be the set of all positive
Borel measures $\lambda$ on $X$ s.t. 
\begin{align}
 & R\left(C\left(X\right)\right)\subset L^{1}\left(\lambda\right),\mbox{ and}\label{eq:m1}\\
 & \lambda\cdot R\ll\lambda\mbox{ (absolutely continuous).}\label{eq:m2}
\end{align}
\end{defn}
\begin{note}
Let $R$ be as above, and set 
\begin{equation}
\mu=\lambda\cdot R,\mbox{ and }W=\frac{d\mu}{d\lambda}=\mbox{Radon-Nikodym derivative, then}\label{eq:m3}
\end{equation}
\begin{equation}
\lambda\left(Rf\right)\underset{\text{defn.}}{=}\underset{\mu\left(f\right)}{\underbrace{\int_{X}Rf\,d\lambda}}=\int_{X}fW\,d\lambda,\quad\forall f\in C\left(X\right).\label{eq:m4}
\end{equation}
$W$ depends on both $R$ and $\lambda$. \end{note}
\begin{defn}
\label{def:cm}For all $x\in X$, let 
\[
\mu_{x}=P\left(\cdot\mid x\right)
\]
be the conditional measure, where 
\begin{equation}
\mu_{x}\left(f\right):=\left(Rf\right)\left(x\right)=\int_{X}f\left(y\right)d\mu_{x}\left(y\right).\label{eq:R}
\end{equation}
\end{defn}
\begin{lem}
Let $X$, and $R$ $\textup{(}$positive in $C\left(X\right)$$\textup{)}$
be as before, then there is a system of measures $P\left(\cdot\mid x\right)$
such that 
\begin{equation}
\left(Rf\right)\left(x\right)=\int_{X}f\left(y\right)P\left(dy\mid x\right),\quad\forall f\in C\left(X\right).\label{eq:cp1}
\end{equation}
\end{lem}
\begin{proof}
Immediate from Riesz' theorem applied to the positive linear functional,
\[
C\left(X\right)\ni f\longrightarrow R\left(f\right)\left(x\right),\quad\forall x\in X.
\]
\end{proof}
\begin{cor}
\label{cor:Rext}$C\left(X\right)\ni f\longrightarrow R\left(f\right)\left(x\right)$
extends to $F\in\mathcal{M}\left(X\right)$, measurable functions
on $X$, s.t. the extended operator $\widetilde{R}$ is as follows:
\begin{equation}
\widetilde{R}\left(F\right)\left(x\right)=\int_{X}F\left(y\right)P\left(dy\mid x\right),\quad F\in\mathcal{M}\left(X\right).\label{eq:Re}
\end{equation}
We will write $R$ also for the extension $\widetilde{R}$. \end{cor}
\begin{rem}
\label{rem:L1}Let $X$, and $R$ be as specified in \defref{R}.
Set 
\[
\mathscr{L}_{1}\left(R\right)=\left\{ \lambda\in\mathscr{L}\left(R\right)\mid\lambda\left(X\right)=1\right\} .
\]
Clearly, $\mathscr{L}_{1}\left(R\right)$ is convex. In this generality,
we address two questions:
\begin{itemize}
\item[Q1.] We show that $\mathscr{L}_{1}\left(R\right)$ is non-empty.
\item[Q2.] What are the extreme points in $\mathscr{L}_{1}\left(R\right)$?
\end{itemize}
\end{rem}
\begin{lem}
Let $\mu_{x}=P\left(\cdot\mid x\right)$ be as above. Let $\lambda\in\mathscr{L}\left(R\right)$,
and let $W$ be the Radon-Nikodym derivative from (\ref{eq:m3}).
Then 
\begin{equation}
\int_{X}P\left(\cdot\mid x\right)d\lambda\left(x\right)=W\left(\cdot\right)d\lambda\left(\cdot\right).\label{eq:c1}
\end{equation}
\end{lem}
\begin{proof}
Immediate from the definition. Indeed, for all Borel subset $E\subset X$,
the following are equivalent ($f=\chi_{E}$): 
\begin{alignat*}{1}
\int Rf\,d\lambda & =\int fW\,d\lambda\\
 & \Updownarrow\\
\int f\left(y\right)P\left(dy\mid x\right)d\lambda\left(x\right) & =\int f\left(x\right)W\left(x\right)d\lambda\left(x\right)\\
 & \Updownarrow\\
\int P\left(E\mid x\right)d\lambda\left(x\right) & =\int_{E}W\left(y\right)d\lambda\left(y\right)\\
 & \Updownarrow\\
\int_{X}P\left(\cdot\mid x\right)d\lambda\left(x\right) & =W\left(\cdot\right)d\lambda\left(\cdot\right)
\end{alignat*}
\end{proof}
\begin{rem}
In general, $\lambda\neq P\left(\cdot\mid x_{0}\right)$, $x_{0}\in X$.
Note that $\lambda=P\left(\cdot\mid x_{0}\right)\in\mathscr{L}\left(R\right)$
iff 
\begin{align}
\int_{y}P\left(\cdot\mid y\right)P\left(dy\mid x_{0}\right) & =W\left(\cdot\right)P\left(\cdot\mid x_{0}\right),\mbox{ i.e., }\nonumber \\
\int_{y}P\left(dz\mid y\right)P\left(dy\mid x_{0}\right) & =W\left(z\right)P\left(dz\mid x_{0}\right)\label{eq:c2}
\end{align}
However, condition (\ref{eq:c2}) is very restrictive, and it is not
satisfied in many cases. See \exaref{IFS} below. \end{rem}
\begin{example}[Iterated Function System (IFS); see e.g., \cite{MR2869044,MR2520021}]
\label{exa:IFS} Let $X=\left[0,1\right]=\mathbb{R}/\mathbb{Z}$,
and $\lambda=$ Lebesgue measure. Fix $v>0$, a positive function
on $\left[0,1\right]$, and set 
\[
\left(Rf\right)\left(x\right)=v\left(\frac{x}{2}\right)f\left(\frac{x}{2}\right)+v\left(\frac{x+1}{2}\right)f\left(\frac{x+1}{2}\right).
\]
Then 
\[
P\left(\cdot\mid x\right)=\underset{\text{atomic}}{\underbrace{v\left(\frac{x}{2}\right)\delta_{\frac{x}{2}}+v\left(\frac{x+1}{2}\right)\delta_{\frac{x+1}{2}}}}\not\ll\underset{\text{non-atomic}}{\underbrace{\vphantom{\left(\frac{x+1}{2}\right)}\lambda}.}
\]
\end{example}
\begin{assumption*}[Additional axiom on $R$]
Let $R$ be the positive mapping in \defref{R}. Assume there exists
$\sigma:X\longrightarrow X$, measurable and \uline{onto}, such
that 
\begin{equation}
R\left(\left(f\circ\sigma\right)g\right)=fRg,\quad\forall f,g\in C\left(X\right).\label{eq:m5}
\end{equation}
$R$ in (\ref{eq:m5}) is a generalized conditional expectation.\end{assumption*}
\begin{lem}
Let $R$ satisfy (\ref{eq:m5}) and let $\left\{ P\left(\cdot\mid x\right)\right\} _{x\in X}$
be the family of conditional measures in \defref{cm}. Then,
\begin{equation}
P\left(E\mid x\right)=\int_{E}\frac{f\left(\sigma\left(y\right)\right)}{f\left(x\right)}P\left(dy\mid x\right)\label{eq:m51}
\end{equation}
for all $f\in C\left(X\right)$, and all $E\in\mathscr{B}\left(X\right)$;
where $\mathscr{B}\left(X\right)$ denotes all Borel subsets of $X$. \end{lem}
\begin{proof}
We have
\begin{align*}
R\left(\left(f\circ\sigma\right)g\right)\left(x\right) & =f\left(x\right)R\left(g\right)\left(x\right)\\
 & \Updownarrow\\
\int f\left(\sigma\left(y\right)\right)g\left(y\right)P\left(dy\mid x\right) & =f\left(x\right)\int g\left(y\right)P\left(dy\mid x\right),\quad\forall f,g\in C\left(X\right),\forall x\in X.\\
 & \Updownarrow\quad g=\chi_{E}\\
\int_{E}f\left(\sigma\left(y\right)\right)P\left(dy\mid x\right) & =f\left(x\right)P\left(E\mid x\right),\quad\forall f\in C\left(X\right),\forall E\in\mathscr{B}\left(X\right),
\end{align*}
and the assertion follows.\end{proof}
\begin{lem}
\label{lem:adjoint}Suppose (\ref{eq:m5}) holds and $\lambda\in\mathscr{L}\left(R\right)$.
Set $W=$ the Radon-Nikodym derivative, then the operator $S:f\longrightarrow Wf\circ\sigma$
is well-defined and linear in $L^{2}\left(\lambda\right)$ with $C\left(X\right)$
as dense domain. In general $S$ is unbounded. Moreover, 
\begin{equation}
S\subset R^{*}\mbox{, containment of unbounded operators,}\label{eq:m6}
\end{equation}
where $R^{*}$ denotes the adjoint operator to $R$, i.e., 
\begin{equation}
\int_{X}\left(Wf\circ\sigma\right)g\,d\lambda=\int_{X}fR\left(g\right)d\lambda;\label{eq:m7}
\end{equation}
holds for all $f,g\in C\left(X\right)$. That is, 
\begin{equation}
R^{*}f=Wf\circ\sigma,\quad\forall f\in C\left(X\right),\label{eq:m71}
\end{equation}
as a weighted composition operator.

Further, the selfadjoint operator $RR^{*}$ is the multiplication
operator: 
\begin{equation}
RR^{*}f=R\left(W\right)f,\quad\forall f\in C\left(X\right);\label{eq:m8}
\end{equation}
i.e., multiplication by the function $R\left(W\right)$. \end{lem}
\begin{proof}
For all $f,g\in C\left(X\right)$, we have 

\begin{eqnarray*}
\int_{X}fR\left(g\right)d\lambda & \underset{\text{by \ensuremath{\left(\ref{eq:m5}\right)}}}{=} & \int_{X}R\left(\left(f\circ\sigma\right)g\right)d\lambda\\
 & \underset{\text{by \ensuremath{\left(\ref{eq:m4}\right)}}}{=} & \int_{X}\underset{=S\left(f\right)}{\underbrace{\left(W\left(f\circ\sigma\right)\right)}}g\,d\lambda,
\end{eqnarray*}
and so (\ref{eq:m6})-(\ref{eq:m7}) follow. Also, 
\[
RR^{*}f\underset{\text{by \ensuremath{\left(\ref{eq:m6}\right)}}}{=}R\left(W\left(f\circ\sigma\right)\right)\underset{\text{by \ensuremath{\left(\ref{eq:m5}\right)}}}{=}fR\left(W\right)=mf,
\]
where $m=R\left(W\right)$. The assertion (\ref{eq:m8}) follows from
this. \end{proof}
\begin{cor}
$S$ is isometric in $L^{2}\left(\lambda\right)$ $\Longleftrightarrow$
$R\left(W\right)=1$. 
\end{cor}

\begin{cor}
$R$ defines a \uline{bounded} operator on $L^{2}\left(\lambda\right)$,
i.e., $L^{2}\left(\lambda\right)\xrightarrow{\;R\;}L^{2}\left(\lambda\right)$
is bounded $\Longleftrightarrow$ $R\left(W\right)\in L^{\infty}\left(\lambda\right)$. \end{cor}
\begin{proof}
Immediate from (\ref{eq:m8}) since 
\begin{equation}
\left\Vert RR^{*}\right\Vert _{2\rightarrow2}=\left\Vert R\right\Vert _{2\rightarrow2}^{2}=\left\Vert R^{*}\right\Vert _{2\rightarrow2}^{2}.\label{eq:m9}
\end{equation}
\end{proof}
\begin{rem}
Let $\lambda\in\mathscr{L}\left(R\right)$, $\mu=\lambda\cdot R$,
and $W=d\mu/d\lambda$ as before. Even if $W\in L^{1}\left(\lambda\right)$,
the following two operators are still well-defined:
\[
L^{2}\left(\lambda\right)\supset\left\{ \begin{split}C\left(X\right) & \ni f\xrightarrow{\quad R\quad}Rf\in L^{\infty}\left(\lambda\right)\subset L^{2}\left(\lambda\right)\\
C\left(X\right) & \ni f\xrightarrow{\quad S\quad}W\left(f\circ\sigma\right)\in L^{2}\left(\lambda\right)
\end{split}
\right\} ,\mbox{ and}
\]
\[
\left\langle Sf,g\right\rangle _{L^{2}}=\left\langle f,Rg\right\rangle _{L^{2}},\quad\forall f,g\in C\left(X\right).
\]
\end{rem}
\begin{cor}
\label{cor:ls}Assume $\frac{1}{W}$ is well-defined. Then $\lambda\circ\sigma^{-1}\ll\lambda$,
and 
\[
\frac{d\lambda\circ\sigma^{-1}}{d\lambda}=R\left(\frac{1}{W}\right),
\]
where $R\left(\frac{1}{W}\right)$ is defined as in (\ref{eq:Re})
of \corref{Rext}.\end{cor}
\begin{proof}
Recall the pull-back measure $\lambda\circ\sigma^{-1}$, where $\sigma^{-1}\left(E\right)=\left\{ z\in X\mid\sigma\left(z\right)\in E\right\} $,
for all Borel sets $E\subset X$. One checks that 
\[
\int f\,d\lambda\circ\sigma^{-1}=\int f\circ\sigma\,d\lambda=\int\frac{1}{W}Wf\circ\sigma\,d\lambda=\int R\left(\frac{1}{W}\right)f\,d\lambda,\quad\forall f\in C\left(X\right);
\]
and the assertion follows.\end{proof}
\begin{cor}
Let $R$, $\lambda$, $W$ be as above, and assume that $\left\Vert R\left(W\right)\right\Vert _{\infty}\leq1$.
Let $h$ be a function on $X$ solving the equation 
\begin{equation}
Rh=h,\quad h\in L^{2}\left(\lambda\right),\mbox{ (\ensuremath{R}-harmonic)}\label{eq:m10}
\end{equation}
then the following implication holds:
\begin{equation}
h\left(x\right)\neq0\Longrightarrow R\left(W\right)\left(x\right)=1.\label{eq:m11}
\end{equation}
\end{cor}
\begin{proof}
By (\ref{eq:m9}), $R$ is contractive, i.e., $\left\Vert R\right\Vert _{2\rightarrow2}=\left\Vert R^{*}\right\Vert _{2\rightarrow2}\leq1$,
$\left\Vert Rf\right\Vert _{L^{2}\left(\lambda\right)}\leq\left\Vert f\right\Vert _{L^{2}\left(\lambda\right)}$;
and so $R^{*}h=h$; and, by (\ref{eq:m8}), 
\begin{equation}
h=h\,R\left(W\right),\;\mbox{pointwise},\label{eq:m12}
\end{equation}
i.e., $h\left(x\right)=h\left(x\right)R\left(W\right)\left(x\right)$,
for all $x\in X$, and (\ref{eq:m11}) follows.\end{proof}
\begin{cor}
Suppose $\lambda\in\mathscr{L}\left(R\right)$ with $R$, $\sigma$,
$W=d\mu/d\lambda$ satisfying the usual axioms, then $\lambda$ is
$\sigma$-invariant, i.e.,
\begin{align}
\int f\circ\sigma\,d\lambda & =\int f\,d\lambda,\quad\forall f\in C\left(X\right)\label{eq:w1}\\
 & \Updownarrow\nonumber \\
\frac{1}{W}\mbox{ exists, and } & R\left(\frac{1}{W}\right)=1\mbox{ on the support of \ensuremath{\lambda}.}\label{eq:w2}
\end{align}
 \end{cor}
\begin{proof}
(\ref{eq:w1})$\Rightarrow$(\ref{eq:w2}) follows from \corref{ls}.
(Also see \corref{Rext}.)

Assume (\ref{eq:w2}), then 
\[
\mbox{LHS}_{\left(\ref{eq:w1}\right)}=\int\frac{1}{W}\underset{\text{\ensuremath{R^{*}f}}}{\underbrace{Wf\circ\sigma}}d\lambda=\int R\left(\frac{1}{W}\right)f\,d\lambda=\int f\,d\lambda,\quad\forall f\in C\left(X\right).
\]
\end{proof}
\begin{cor}
Let $X$, $R$, $\lambda$, $W$, $\sigma$ be as specified above.
Recall that $f\geq0$ $\Longrightarrow$ $Rf\geq0$, and $R\left(\left(f\circ\sigma\right)g\right)=fR\left(g\right)$,
$\forall f,g\in C\left(X\right)$. Assume further that $W\in L^{2}\left(\lambda\right)$,
then 
\begin{align}
\int_{X}\left|W\right|^{2}f\circ\sigma\,d\lambda & =0\mbox{, for some \ensuremath{f\in C\left(X\right)}}\label{eq:w3}\\
 & \Updownarrow\nonumber \\
\int_{X}f\,R\left(W\right)d\lambda & =0.\label{eq:w4}
\end{align}
\end{cor}
\begin{proof}
Use that $L^{2}\left(\lambda\right)\ni f\xrightarrow{\;R^{*}\;}Wf\circ\sigma\in L^{2}\left(\lambda\right)$,
we conclude that 
\begin{align}
\int_{X}\left|W\right|^{2}f\circ\sigma\,d\lambda & =\int_{X}Wf\circ\sigma\overline{W}\,d\lambda=\left\langle R^{*}f,W\right\rangle _{L^{2}\left(\lambda\right)}\label{eq:w31}\\
 & =\left\langle f,R\left(W\right)\right\rangle _{L^{2}\left(\lambda\right)}=\int_{X}f\left(x\right)R(\overline{W})\left(x\right)d\lambda\left(x\right).\nonumber 
\end{align}
\end{proof}
\begin{cor}
Let $X$, $R$, $\lambda$, $W$, $\sigma$ be as above, and let $E\subset X$
be a Borel set; then 
\begin{equation}
\int_{\sigma^{-1}\left(E\right)}\left|W\right|^{2}d\lambda=\int_{E}R\left(W\right)d\lambda,\label{eq:w5}
\end{equation}
and so in particular, $R\left(W\right)\geq0$ a.e. on $X$ w.r.t.
$\lambda$. \end{cor}
\begin{proof}
Approximate $\chi_{E}$ with $f\in C\left(X\right)$ and use (\ref{eq:w31}),
we have 
\[
\int\left|W\right|^{2}f\circ\sigma\,d\lambda=\int R\left(W\right)f\,d\lambda,
\]
which is (\ref{eq:w5}).\end{proof}
\begin{example}
\label{exa:cc1} Let $X=\left[0,1\right]=\mathbb{R}/\mathbb{Z}$,
and $d\lambda=dx=$ Lebesgue measure. Fix $W>0$, a positive function
over $\left[0,1\right]$, and set 
\begin{equation}
\left(Rh\right)\left(x\right)=\frac{1}{2}\left(W\left(\frac{x}{2}\right)h\left(\frac{x}{2}\right)+W\left(\frac{x+1}{2}\right)h\left(\frac{x+1}{2}\right)\right).\label{eq:cc2}
\end{equation}
Let $\sigma\left(x\right)=2x\mod1$, $x\in X$, then
\begin{equation}
\int_{0}^{1}g\left(x\right)\left(Rh\right)\left(x\right)dx=\int_{0}^{1}W\left(x\right)g\left(\sigma\left(x\right)\right)h\left(x\right)dx,\quad\forall f,g\in C\left(X\right).\label{eq:cc3}
\end{equation}
\end{example}
\begin{proof}
We introduce the mappings $\tau_{0}$ and $\tau_{1}$, as in Fig \ref{fig:sigma}-\ref{fig:tau},
so that $\sigma\left(\tau_{i}\left(x\right)\right)=x$, for all $x\in X$,
$i=0,1$. One checks that 
\begin{equation}
R\left(\left(g\circ\sigma\right)h\right)\left(x\right)=g\left(x\right)\left(Rh\right)\left(x\right).\label{eq:cc4}
\end{equation}
Note that $\lambda\in\mathscr{L}\left(R\right)$. Indeed, we have
\begin{eqnarray*}
\lambda\left(Rh\right) & = & \int_{0}^{1}\left(Rh\right)\left(x\right)dx\\
 & \underset{\text{by \ensuremath{\left(\ref{eq:cc2}\right)}}}{=} & \int_{0}^{1}\frac{1}{2}\left(\left(Wh\right)\left(\frac{x}{2}\right)+\left(Wh\right)\left(\frac{x+1}{2}\right)\right)dx\\
 & = & \int_{0}^{1}\left(Wh\right)\left(x\right)dx=\int_{0}^{1}h\left(x\right)W\left(x\right)dx,
\end{eqnarray*}
and so $\frac{d\mu}{d\lambda}\left(x\right)=W\left(x\right)$, where
$\mu=\lambda\cdot R$. \end{proof}
\begin{example}
\label{exa:R1}Let $R$ be as in (\ref{eq:cc2}), and let $h$ be
an $R$-harmonic function, i.e., 
\begin{equation}
\left(Rh\right)\left(x\right)=\frac{1}{2}\left(\left(Wh\right)\left(\frac{x}{2}\right)+\left(Wh\right)\left(\frac{x+1}{2}\right)\right)=h\left(x\right),\quad x\in X=\left[0,1\right].\label{eq:cc5}
\end{equation}
Setting $\widehat{h}\left(n\right)=\int_{0}^{1}e\left(nx\right)h\left(x\right)dx$,
with $e\left(nx\right):=e^{i2\pi nx}$, it follows from (\ref{eq:cc5})
that 
\begin{align*}
\widehat{h}\left(n\right) & =\int_{0}^{1}e\left(nx\right)\left(Rh\right)\left(x\right)dx\\
 & =\int_{0}^{1}W\left(x\right)e\left(2nx\right)h\left(x\right)dx=\left(Wh\right)^{\wedge}\left(2n\right),\quad\forall n\in\mathbb{Z}.
\end{align*}
An iteration gives 
\begin{align*}
\widehat{h}\left(n\right) & =\int_{0}^{1}W\left(x\right)e\left(2nx\right)\left(Rh\right)\left(x\right)dx\\
 & =\int_{0}^{1}W\left(x\right)W\left(2x\right)e\left(2^{2}nx\right)h\left(x\right)dx\\
 & \cdots\\
 & =\int_{0}^{1}W\left(x\right)W\left(2x\right)\cdots W\left(2^{k-1}x\right)e\left(2^{k}nx\right)h\left(x\right)dx,
\end{align*}
and so 
\[
\widehat{h}\left(n\right)=\left(W_{k}h\right)^{\wedge}\left(2^{k}n\right),\quad\forall n\in\mathbb{Z},\:\forall k=0,1,2,\cdots;
\]
where $W_{k}\left(x\right):=W\left(x\right)W\left(2x\right)\cdots W\left(2^{k-1}x\right)$. 
\end{example}
\begin{figure}[H]
\includegraphics[width=0.3\textwidth]{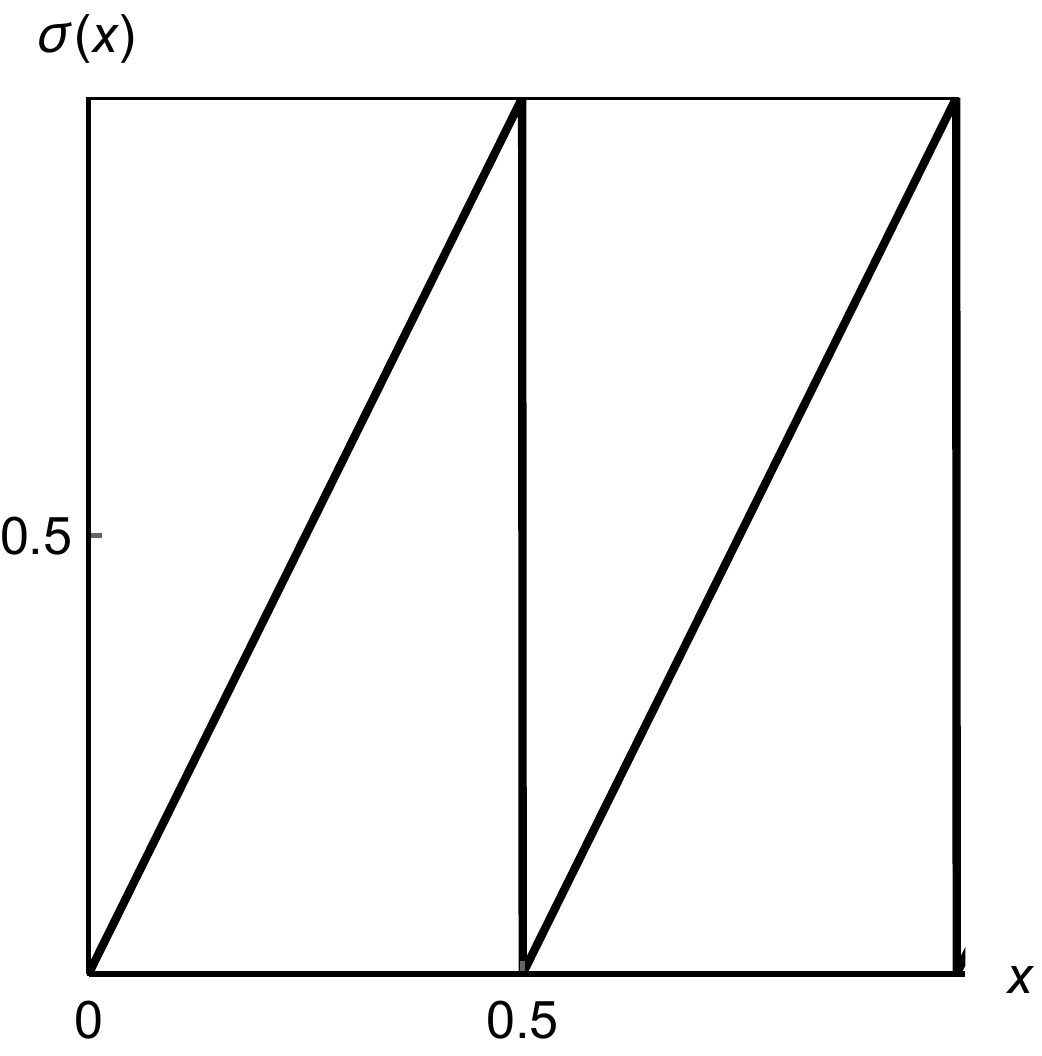}

\caption{\label{fig:sigma}$\sigma\left(x\right)=2x\mod1$}

\end{figure}

\begin{figure}[H]
\subfloat[]{\includegraphics[width=0.3\textwidth]{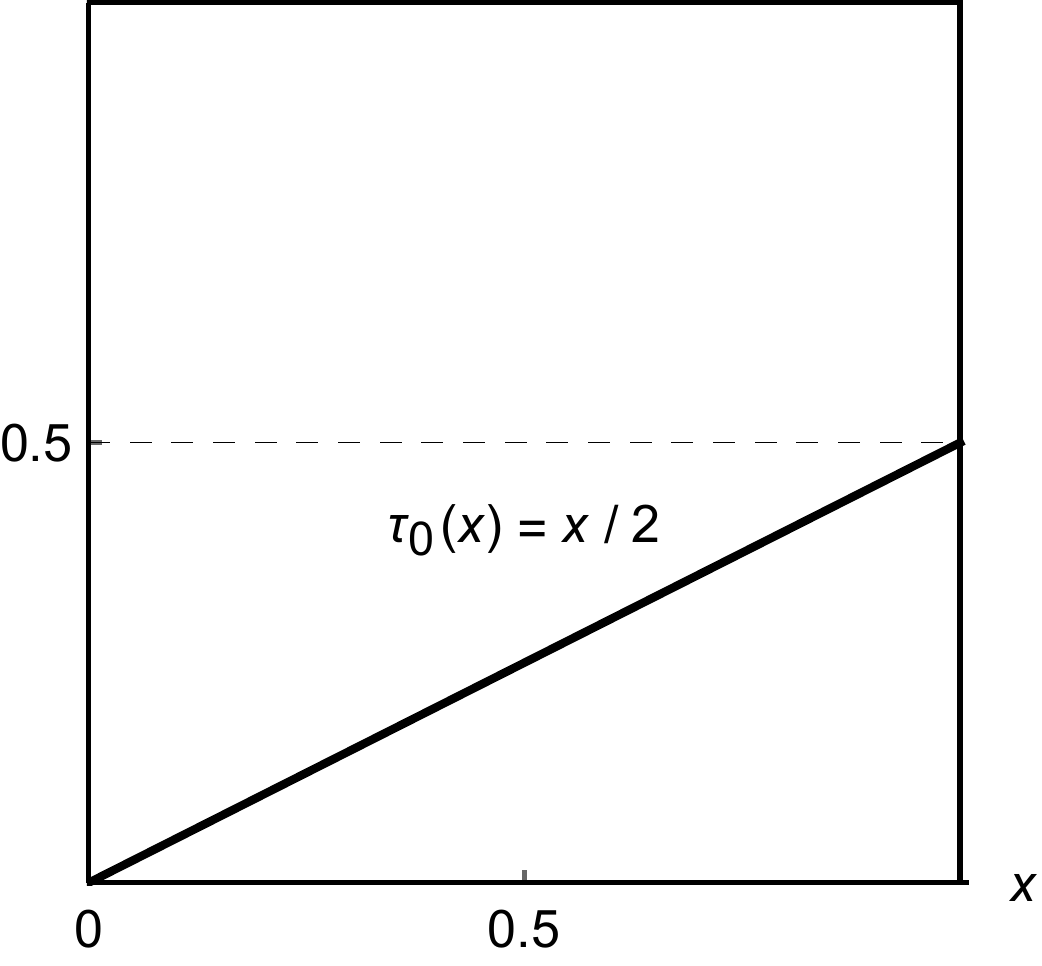}

}\hspace{1.5cm}\subfloat[]{\includegraphics[width=0.3\textwidth]{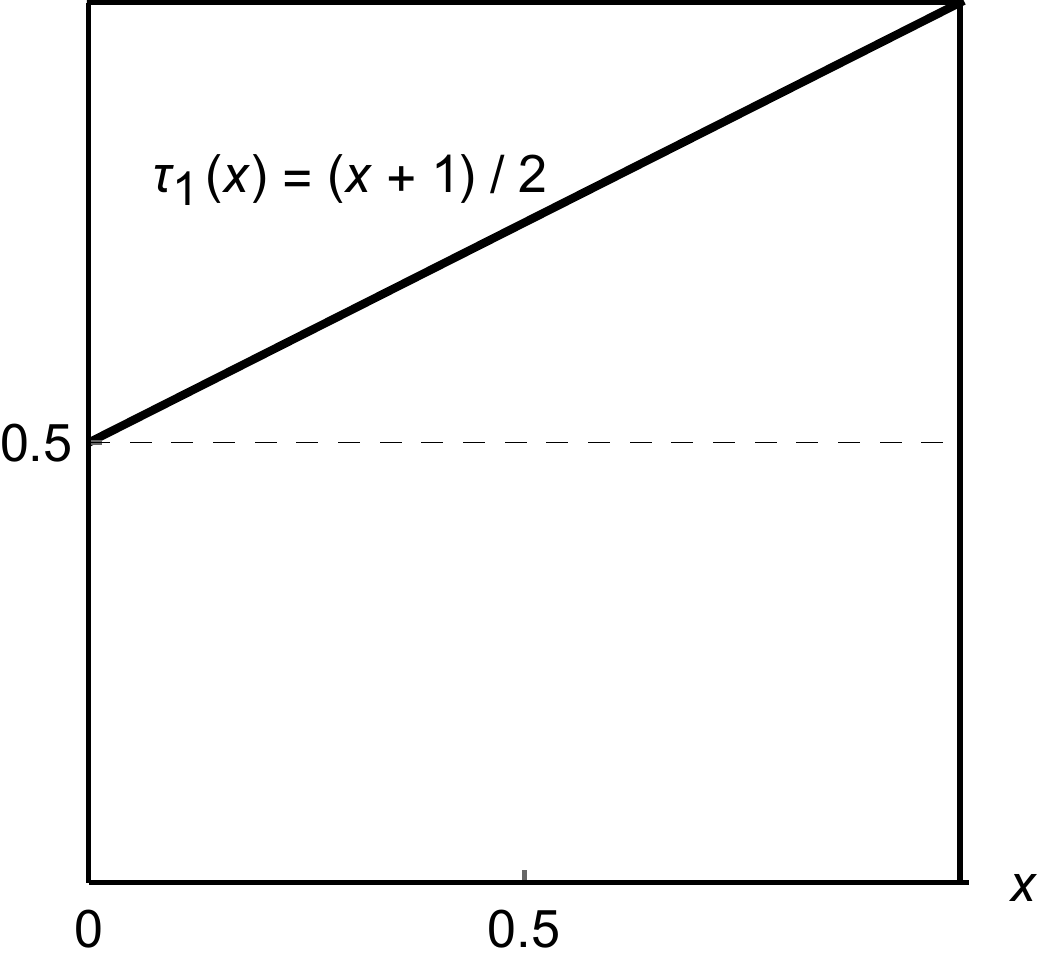}

}

\caption{\label{fig:tau}$\tau_{0}=\frac{x}{2}$, $\tau_{1}=\frac{x+1}{2}$}
\end{figure}

\section{\label{sec:IFS}Iterated Function Systems: The General Case}

In this section we discuss a subfamily of systems where the generalized
transfer operator $R$ is associated with an Iterated Function system
(IFS).

Let $X$ be a compact Hausdorff space, $n\in\mathbb{N}$, and let
\begin{equation}
\tau_{i}:X\longrightarrow X,\quad1\leq i\leq n\label{eq:f1}
\end{equation}
be a system of endomorphisms. Let 
\begin{equation}
p_{i}>0,\;s.t.\;\sum_{i=1}^{n}p_{i}=1.\label{eq:f2}
\end{equation}
Following \cite{MR625600,MR2869044,MR2560042,MR2529881,MR2525528,MR2520021},
we say that (\ref{eq:f1})-(\ref{eq:f2}) is an \emph{Iterated Function
System }(IFS) if there is a Borel probability measure $\lambda$ on
$X$ such that 
\begin{equation}
\sum_{i=1}^{n}p_{i}\int_{X}f\left(\tau_{i}\left(x\right)\right)d\lambda\left(x\right)=\int_{X}f\left(x\right)d\lambda\left(x\right)\label{eq:f3}
\end{equation}
holds for all $f\in C\left(X\right)$. Note that (\ref{eq:f3}) may
also be expressed as follows:
\begin{equation}
\sum_{i}p_{i}\,\lambda\circ\tau_{i}^{-1}=\lambda.\label{eq:f4}
\end{equation}
The measure $\lambda$ is called an IFS measure. 

Let $W\in L^{1}\left(\lambda\right)$, $W\geq0$, and set 
\begin{equation}
\left(R_{W}f\right)\left(x\right)=\sum_{i=1}^{n}p_{i}\left(Wf\right)\left(\tau_{i}\left(x\right)\right),\quad x\in X,\:f\in C\left(X\right),\label{eq:f5}
\end{equation}
where $\left(Wf\right)\left(\tau_{i}\left(x\right)\right):=W\left(\tau_{i}\left(x\right)\right)f\left(\tau_{i}\left(x\right)\right)$. 
\begin{lem}
If $W$ is as above, and if $\lambda$ is an IFS measure, then $\lambda\in\mathscr{L}_{1}\left(R_{W}\right)$,
see \remref{L1}. \end{lem}
\begin{proof}
We establish the conclusion by verifying that, under the assumptions,
we have 
\begin{equation}
\int_{X}\left(R_{W}f\right)\left(x\right)d\lambda\left(x\right)=\int_{X}W\left(x\right)f\left(x\right)d\lambda\left(x\right),\quad\forall f\in C\left(X\right),\label{eq:f6}
\end{equation}
i.e., $W$ is the Radon-Nikodym derivative, $d\mu_{W}/d\lambda=W$,
where $\mu_{W}=\lambda\cdot R_{W}$. Indeed, 
\begin{eqnarray*}
\mbox{LHS}_{\left(\ref{eq:f6}\right)} & \underset{\text{by \ensuremath{\left(\ref{eq:f5}\right)}}}{=} & \sum_{i=1}^{n}p_{i}\int_{X}\left(Wf\right)\left(\tau_{i}\left(x\right)\right)d\lambda\left(x\right)\\
 & \underset{\text{by \ensuremath{\left(\ref{eq:f3}\right)}}}{=} & \int_{X}\left(Wf\right)\left(x\right)d\lambda\left(x\right)=\mbox{RHS}_{\left(\ref{eq:f6}\right).}
\end{eqnarray*}
\end{proof}
\begin{rem}
The setting of \exaref{R1}, we have an IFS corresponding to the two
mappings in \figref{tau}, and, in this setting, the corresponding
IFS measure $\lambda$ on the unit interval $X=\left[0,1\right]$
can then easily be checked to be the restriction to $\left[0,1\right]$
of the standard Lebesgue measure. It is important to mention that
there is a rich literature on IFS measures, see e.g., \cite{MR625600,MR2869044,MR2560042,MR2529881,MR2525528,MR2520021},
and the variety of IFS measures associated to function systems includes
explicit classes measures of fractal dimension. 
\end{rem}

\section{\label{sec:L1}The Set $\mathscr{L}_{1}\left(R\right)$ from a Quadratic
Estimate}

In order to build a path-space probability space from a given generalized
transfer operator $R$, a certain spectral property for $R$ must
be satisfied, and we discuss this below; see \thmref{mc}. The statement
of the problem requires the introduction of a Hilbert space of sigma
functions, also called square densities.

Let $X$ be a locally compact Hausdorff space, and let $R:C\left(X\right)\longrightarrow\mathcal{M}\left(X\right)$
be given, subject to the conditions in \defref{R} and \remref{L1}.

For every probability measure $\lambda$ on $X$, we apply the Radon-Nikodym
decomposition (see \cite{MR924157}) to the measure $\lambda R$,
getting 
\begin{equation}
\lambda R=\mu_{abs}+\mu_{sing}\label{eq:L1}
\end{equation}
where the two terms on the RHS in (\ref{eq:L1}) are absolutely continuous
w.r.t $\lambda$, respectively, with $\mu_{sing}$ and $\lambda$
mutually singular. Hence there is a positive $W_{\lambda}\in\mathscr{L}_{1}\left(\lambda\right)$
such that $\mu_{abs}=W_{\lambda}d\lambda$. When $\lambda$ is fixed,
set 
\begin{equation}
\tau=\frac{1}{2}\left(\mu_{abs}+\lambda R\right).\label{eq:L2}
\end{equation}
We have the following:
\begin{thm}
\label{thm:mc}Let $\left(X,R\right)$ be as described above, and
let $Prob\left(X\right)$ be the convex set of all probability measures
on $X$. For $\lambda\in Prob\left(X\right)$, let $\tau$ be the
corresponding measure given by (\ref{eq:L2}). Then $\mathscr{L}_{1}\left(R\right)\neq0$
if and only if 
\begin{equation}
\inf_{\lambda\in Prob\left(X\right)}\int_{X}\left|\sqrt{\frac{d\left(\lambda R\right)}{d\tau}}-W_{\lambda}\sqrt{\frac{d\lambda}{d\tau}}\right|^{2}d\tau=0.\label{eq:L3}
\end{equation}
\end{thm}
\begin{proof}
To carry out the proof details, we shall make use of the Hilbert space
$Sig\left(X\right)$ of sigma-functions on $X$. While it has been
used in, for example \cite{MR0282379,MR0016534,MR562914,MR2847247},
we shall introduce the basic facts which will be needed. 

Elements in $Sig\left(X\right)$ are equivalence classes of pairs
$\left(f,\mu\right)$, where $f\in L^{2}\left(\mu\right)$, and $\mu$
is a positive finite measure on $X$; we say that $\left(f,\mu\right)\sim\left(g,\nu\right)$
for two such pairs iff $\tau=\frac{1}{2}\left(\mu+\nu\right)$ satisfies
\begin{equation}
f\sqrt{\frac{d\mu}{d\tau}}=g\sqrt{\frac{d\nu}{d\tau}}\quad\mbox{a.e. on \ensuremath{X} w.r.t. \ensuremath{\tau}.}\label{eq:L4}
\end{equation}

If $class\left(f_{i},\mu_{i}\right)$, $i=1,2$, are two equivalence
classes, then the operations in $Sig\left(X\right)$ are as follows:
First set $\tau_{s}=\frac{1}{2}\left(\mu_{1}+\mu_{2}\right)$, then
the inner product in $Sig\left(X\right)$ is 
\[
\int_{X}\overline{f_{1}}\sqrt{\frac{d\mu_{1}}{d\tau_{s}}}f_{2}\sqrt{\frac{d\mu_{2}}{d\tau_{s}}}d\tau_{s},
\]
and the sum is 
\[
class\left(f_{1}\sqrt{\frac{d\mu_{1}}{d\tau_{s}}}+f_{2}\sqrt{\frac{d\mu_{2}}{d\tau_{s}}},\tau_{s}\right).
\]
It is known that these definitions pass to equivalence classes; and
that $Sig\left(X\right)$ is a Hilbert space; in particular, it is
complete.

In order to complete the proof of the theorem, we shall need the following
facts about the Hilbert space $Sig\left(X\right)$; see e.g., \cite{MR0282379}:
First some notation; we set 
\begin{equation}
f\sqrt{d\mu}=class\left(f,\mu\right)\in Sig\left(X\right);\label{eq:L5}
\end{equation}
and when $\mu$ is fixed, we set $\mathfrak{M}_{2}\left(\mu\right)$
to be the closed subspace in $Sig\left(X\right)$ spanned by 
\[
\left\{ f\sqrt{d\mu}\mid f\in L^{2}\left(\mu\right)\right\} .
\]

We then have:
\begin{equation}
\left\Vert f\sqrt{d\mu}\right\Vert _{Sig\left(X\right)}^{2}=\left\Vert f\right\Vert _{L^{2}\left(\mu\right)}^{2}=\int_{X}\left|f\right|^{2}d\mu;\label{eq:L6}
\end{equation}
and so, in particular, 
\begin{equation}
L^{2}\left(\mu\right)\ni f\longmapsto f\sqrt{d\mu}\in Sig\left(X\right)\label{eq:L7}
\end{equation}
defines an isometry with range $\mathfrak{M}_{2}\left(\mu\right)$.
We shall abbreviate $\sqrt{d\mu}$ as $\sqrt{\mu}$. For two measures
$\mu$ and $\nu$, the following three facts holds:
\begin{eqnarray}
\left[\mu\ll\nu\right] & \Longleftrightarrow & \mathfrak{M}_{2}\left(\mu\right)\subseteq\mathfrak{M}_{2}\left(\nu\right),\nonumber \\
\left[\mu\approx\nu\right] & \Longleftrightarrow & \mathfrak{M}_{2}\left(\mu\right)=\mathfrak{M}_{2}\left(\nu\right)\mbox{, and}\label{eq:L8}\\
\left[\mbox{\ensuremath{\begin{matrix}\mbox{\ensuremath{\mu}\ and \ensuremath{\nu} are }\\
 \mbox{mutually singular} 
\end{matrix}}}\right] & \Longleftrightarrow & \mathfrak{M}_{2}\left(\mu\right)\perp\mathfrak{M}_{2}\left(\nu\right).\nonumber 
\end{eqnarray}

As a result, we note that therefore, the decomposition in (\ref{eq:L1})
is orthogonal in $Sig\left(X\right)$, and further that a fixed $\lambda\in Prob\left(X\right)$
is in $\mathscr{L}_{1}\left(R\right)$ if and only if
\begin{gather}
\mathfrak{M}_{2}\left(\lambda R\right)\subseteq\mathfrak{M}_{2}\left(\lambda\right)\label{eq:L9}\\
\Updownarrow\nonumber \\
\inf_{\lambda\in Prob\left(X\right)}\left\Vert \sqrt{\lambda R}-W_{\lambda}\sqrt{\lambda}\right\Vert _{Sig\left(X\right)}^{2}=0\label{eq:L10}
\end{gather}
Moreover, (\ref{eq:L9}) is a restatement of (\ref{eq:L3}).

We now turn to the conclusions in the theorem: One implication is
clear. If now the infimum in (\ref{eq:L10}) is zero, then there is
a sequence $\left\{ \lambda_{n}\right\} \subset Prob\left(X\right)$
such that 
\begin{equation}
\lim_{n}\left\Vert \sqrt{\lambda_{n}R}-W_{\lambda_{n}}\sqrt{\lambda_{n}}\right\Vert _{Sig\left(X\right)}^{2}=0.\label{eq:L11}
\end{equation}
Combining (\ref{eq:L10}) and (\ref{eq:L11}), and possibly passing
to a subsequence, we conclude that there is a sequence $W_{\lambda_{n}}\sqrt{\lambda_{n}}$
which is convergent in $Sig\left(X\right)$. Let the limit be $W_{\lambda_{0}}\sqrt{\lambda_{0}}$,
and it follows that $\lambda_{0}\in\mathscr{L}_{1}\left(R\right)$. \end{proof}
\begin{rem}
Since $Sig\left(X\right)$ is a Hilbert space, we conclude that the
sequence $\left\{ W_{\lambda_{n}}\sqrt{\lambda_{n}}\right\} _{n}$
in $Sig\left(X\right)$ satisfies 
\[
\lim_{n}\left\Vert \sqrt{\lambda_{0}R}-W_{\lambda_{n}}\sqrt{\lambda_{n}}\right\Vert _{Sig\left(X\right)}^{2}=0;
\]
where the desired measure $\lambda_{0}\in\mathscr{L}_{1}\left(R\right)$
may be taken to be
\[
d\lambda_{0}\left(\cdot\cdot\right)=\sum_{n=1}^{\infty}\frac{1}{2^{n}}\frac{d\lambda_{n}\left(\cdot\cdot\right)}{\lambda_{n}\left(X\right)}.
\]

\end{rem}

\section{\label{sec:EA}From Endomorphism to Automorphism}

In this section (\thmref{sol}), we build a path-space probability
space from a given generalized transfer operator $R$ assumed to satisfy
the spectral property from above.

There is a generalized family of multi-resolution measures on \emph{solenoids},
and we  shall need the following facts (see e.g., \cite{MR625600,MR2869044,MR2560042,MR2529881,MR2525528,MR2520021}): 

Let $X$ be a compact Hausdorff space, and let $\sigma:X\rightarrow X$
be a continuous endomorphism \emph{onto $X$. }Let 
\begin{equation}
\Omega:=\prod_{0}^{\infty}X=X\times X\times\cdots\label{eq:s1}
\end{equation}
be the infinite Cartesian product with coordinate mappings $Z_{n}:\Omega\longrightarrow X$,
\begin{equation}
Z_{n}\left(x_{0},x_{1},x_{2},\cdots\right)=x_{n}\in X,\quad n\in0,1,2,\cdots.\label{eq:s2}
\end{equation}

The associated \emph{solenoid }$Sol_{\sigma}\left(X\right)\subset\prod_{0}^{\infty}X$
is defined as follows: 
\begin{equation}
Sol_{\sigma}\left(X\right)=\left\{ \left(x_{n}\right)_{0}^{\infty}\in\Omega\mid\sigma\left(x_{n+1}\right)=x_{n},\;n=0,1,2,\cdots\right\} ;\label{eq:s3}
\end{equation}
and set 
\begin{equation}
\widetilde{\sigma}\left(x_{0},x_{1},x_{2},\cdots\right):=\left(\sigma\left(x_{0}\right),x_{0},x_{1},x_{2},\cdots\right).\label{eq:s4}
\end{equation}

We give $Sol_{\sigma}\left(X\right)$ its relative projective topology,
and note that the restricted random variable $\left(Z_{n}\right)_{n=0}^{\infty}$
from (\ref{eq:s2}) are then continuous. Moreover $\widetilde{\sigma}$,
in (\ref{eq:s4}), is invertible with 
\begin{equation}
\widetilde{\sigma}^{-1}\left(x_{0},x_{1},x_{2},x_{3},\cdots\right)=\left(x_{1},x_{2},x_{3},\cdots\right),\label{eq:s5}
\end{equation}
\begin{equation}
\widetilde{\sigma}\widetilde{\sigma}^{-1}=\widetilde{\sigma}^{-1}\widetilde{\sigma}=Id_{Sol_{\sigma}\left(X\right)}.\label{eq:s6}
\end{equation}

Let $\left(X,R,\lambda,W\right)$ be as specified in \secref{setting}.
In particular, $R$ is positive, i.e., $f\in C\left(X\right)$, $f\geq0$
$\Longrightarrow$ $R\left(f\right)\geq0$, and 
\begin{equation}
R\left(\left(f\circ\sigma\right)g\right)=fR\left(g\right),\quad\forall f,g\in C\left(X\right).\label{eq:s7}
\end{equation}
Moreover, $W$ is the Radon-Nikodym derivative of the measure $f\longrightarrow\lambda\left(R\left(f\right)\right)$
w.r.t. $\lambda$, i.e., 
\begin{equation}
\int_{X}R\left(f\right)d\lambda=\int_{X}fW\,d\lambda,\quad\forall f\in C\left(X\right).\label{eq:s8}
\end{equation}
Let $h\in L^{\infty}\left(\lambda\right)$, $h\geq0$, satisfying
\begin{equation}
Rh=h,\mbox{ and }\int_{X}h\,d\lambda=1.\label{eq:s9}
\end{equation}

\begin{rem}
In view of equations (\ref{eq:c1}) and (\ref{eq:s9}), it is natural
to think of these conditions as a generalized Perron-Frobenius property
for $R$.\end{rem}
\begin{thm}
\label{thm:sol}With the assumptions (\ref{eq:s7})-(\ref{eq:s9}),
we have the following conclusions:
\begin{enumerate}
\item For every $x\in X$, there is a unique Borel probability measure $\mathbb{P}_{x}$
on $Sol_{\sigma}\left(X\right)$ such that for all $n$, and all $f_{0},f_{1},\cdots,f_{n}\in C\left(X\right)$,
\begin{eqnarray}
 &  & \int_{Z_{0}^{-1}\left(x\right)}\left(f_{0}\circ Z_{0}\right)\left(f_{1}\circ Z_{1}\right)\cdots\left(f_{n}\circ Z_{n}\right)d\mathbb{P}_{x}\nonumber \\
 & = & f_{0}\left(x\right)R\left(f_{1}R\left(f_{2}R\left(\cdots R\left(f_{n}h\right)\cdots\right)\right)\right)\left(x\right).\label{eq:s10}
\end{eqnarray}

\item Moreover, setting 
\begin{equation}
\mathbb{P}=\int_{X}\mathbb{P}_{x}\,d\lambda\left(x\right),\label{eq:s11}
\end{equation}
we get that $\mathbb{P}$ is a probability measure on $Sol_{\sigma}\left(X\right)$
such that 
\begin{equation}
\mathbb{E}_{\mathbb{P}}\left(\cdots\mid Z_{0}=x\right)=\mathbb{P}_{x}\label{eq:s12}
\end{equation}
where the LHS in (\ref{eq:s12}) is the conditional measure, and the
RHS is the measure from (\ref{eq:s10}). 
\item \label{enu:sol3}We have the following Radon-Nikodym derivative:
\begin{equation}
\frac{d\mathbb{P}\circ\widetilde{\sigma}}{d\mathbb{P}}=W\circ Z_{0},\label{eq:s13}
\end{equation}
as an identity of the two functions specified in (\ref{eq:s13}).
Equivalently, setting 
\[
U\psi=\left(\sqrt{W\circ Z_{0}}\right)\psi\circ\widetilde{\sigma},\quad\psi\in L^{2}\left(Sol_{\sigma}\left(X\right),\mathbb{P}\right),
\]
then $U$ is a unitary operator in $L^{2}\left(Sol_{\sigma}\left(X\right),\mathbb{P}\right)$.
\end{enumerate}
\end{thm}
\begin{proof}
This is the basic Kolmogorov inductive limit construction. We note
that, by Stone-Weierstrass, the space of cylinder-functions
\begin{equation}
\left(f_{0}\circ Z_{0}\right)\left(f_{1}\circ Z_{1}\right)\cdots\left(f_{n}\circ Z_{n}\right)\label{eq:sp1}
\end{equation}
is dense in $C\left(Sol_{\sigma}\left(X\right)\right)$. Fix $x\in X$,
and start with $Z_{0}^{-1}\left(x\right)$, set 
\begin{equation}
L_{n}^{x}\left(f_{1},f_{2},\cdots,f_{n}\right)=R\left(f_{1}R\left(f_{2}\cdots R\left(f_{n}h\right)\cdots\right)\right)\left(x\right).\label{eq:sp2}
\end{equation}
We get the desired consistency: 
\begin{equation}
L_{n+1}^{x}\left(f_{1},f_{2},\cdots,f_{n},\mathbbm{1}\right)=L_{n}^{x}\left(f_{1},f_{2},\cdots,f_{n}\right)\label{eq:sp3}
\end{equation}
where $\mathbbm{1}$ denotes the constant function $1$ on $X$. Indeed,
\begin{align*}
R\left(f_{n-1}R\left(f_{n}R\left(\mathbbm{1}h\right)\right)\right)\left(x\right) & =R\left(f_{n-1}R\left(f_{n}R\left(h\right)\right)\right)\left(x\right)\\
 & =R\left(f_{n-1}R\left(f_{n}h\right)\right)\left(x\right),\quad(\mbox{by }\left(\ref{eq:s9}\right))
\end{align*}
as claimed in (\ref{eq:sp3}).\end{proof}
\begin{lem}
Let $R$, $X$, $P\left(\cdot\mid x\right)$ be as above. Assume $h\geq0$
on $X$, and $Rh=h$. Then 
\begin{equation}
\left|R\left(fh\right)\left(x\right)\right|\leq\left\Vert f\right\Vert _{\infty}h\left(x\right)\label{eq:cp3}
\end{equation}
where $\left\Vert f\right\Vert _{\infty}$ is the $P\left(\cdot\mid x\right)$
$L^{\infty}$-norm on functions on $X$.\end{lem}
\begin{proof}
We may apply Cauchy-Schwarz to $P\left(\cdot\mid x\right)$ in a sequence
of steps as follows:
\begin{alignat}{2}
\left|R\left(fh\right)\left(x\right)\right| & =\left|R(fh^{\frac{1}{2}}h^{\frac{1}{2}})\left(x\right)\right|\label{eq:cp4}\\
 & \leq(R(\left|f\right|^{2}h)\left(x\right))^{\frac{1}{2}}(R\left(h\right)\left(x\right))^{\frac{1}{2}} & \qquad & \mbox{by Schwarz}\nonumber \\
 & =(R(\left|f\right|^{2}h)\left(x\right))^{\frac{1}{2}}h\left(x\right)^{\frac{1}{2}} &  & \mbox{since }Rh=h\nonumber \\
 & \leq\underset{\longrightarrow\left\Vert f\right\Vert _{\infty}}{\underbrace{R(\left|f\right|^{2^{p}}h)\left(x\right)^{\frac{1}{2^{p}}}}}\:\underset{\longrightarrow h\left(x\right)}{\underbrace{h\left(x\right)^{\frac{1}{2}+\frac{1}{2^{2}}+\cdots+\frac{1}{2^{p}}}}} &  & \mbox{by induction, and let \ensuremath{p\longrightarrow\infty}}\nonumber 
\end{alignat}
An elementary result in measure theory (see \cite{MR924157}) shows
that 
\begin{equation}
\lim_{p\rightarrow\infty}R(\left|f\right|^{2^{p}}h)\left(x\right)^{\frac{1}{2^{p}}}=\left\Vert f\right\Vert _{\infty};\label{eq:cp5}
\end{equation}
and so the desired estimate (\ref{eq:cp3}) holds.\end{proof}
\begin{cor}
\label{cor:pm1}Let $R,h,\sigma,W,\lambda\in\mathscr{L}\left(R\right)$
be as above, where $\mu=\lambda\cdot R$, and $W=\frac{d\mu}{d\lambda}$.
Assume $h>0$ on $X$, and $Rh=h$. Let $\mathbb{P}$ and $\mathbb{P}_{x}$
be the measures on $Sol_{\sigma}\left(X\right)$ as in \thmref{sol},
where $\mathbb{P}_{x}$ is determined by 
\begin{align}
\int_{Z_{0}^{-1}\left(x\right)}\left(f_{1}\circ Z_{1}\right)\cdots\left(f_{n}\circ Z_{n}\right)d\mathbb{P}_{x} & =L_{n}^{x}\left(f_{1},\cdots,f_{n}\right)\nonumber \\
 & =R\left(f_{1}R\left(f_{2}\cdots R\left(f_{n}h\right)\cdots\right)\right)\left(x\right).\label{eq:cp5-1}
\end{align}
Then, 
\begin{equation}
\left|\int\psi d\mathbb{P}_{x}\right|\leq\left\Vert \psi\right\Vert _{\infty}h\left(x\right),\quad\forall\psi.
\end{equation}
That is, 
\begin{equation}
\left|\frac{\mathbb{E}\left(\psi\mid x\right)}{h\left(x\right)}\right|\leq\left\Vert \psi\right\Vert _{\infty},\quad\forall\psi;
\end{equation}
in particular, $\mathbb{E}\left(\psi\mid x\right)=\int\psi d\mathbb{P}_{x}\left(=\mbox{conditional expectation}\right)$
is well-defined. See (\ref{eq:s12}).
\end{cor}

\begin{cor}
\label{cor:cp}Let $X,R,\sigma,\lambda,h$ be as described above;
in particular, $Rh=h$ is assumed. Let $\left\{ \mathbb{P}_{x}\right\} _{x\in X}$
be the measures from \corref{pm1}. Then 
\[
h\left(x\right)=\mathbb{P}_{x}\left(Z_{0}^{-1}\left(x\right)\right)\mbox{ for all \ensuremath{x\in X}.}
\]

\end{cor}

\begin{cor}
Let $R,h,\sigma,W,\lambda$ be as above, and let $\mathbb{P}$ and
$\mathbb{P}_{x}$ be the corresponding measures on $Sol_{\sigma}\left(X\right)$,
then $V_{0}$ is \uline{isometric}, where 
\[
V_{0}:L^{2}\left(X,h\,d\lambda\right)\longrightarrow L^{2}\left(Sol_{\sigma}\left(X\right),\mathbb{P}\right)
\]
is given by 
\begin{equation}
V_{0}g=g\circ Z_{0},\quad g\in L^{2}\left(X,h\,d\lambda\right),\label{eq:cp-v0a}
\end{equation}
and 
\begin{equation}
\left(V^{*}\psi\right)\left(x\right)=\frac{\mathbb{E}\left(\psi\mid x\right)}{h\left(x\right)},\quad\forall x\in X,\:\forall\psi\in L^{2}\left(Sol_{\sigma}\left(X\right),\mathbb{P}\right).\label{eq:cp-v0b}
\end{equation}
\end{cor}
\begin{proof}
Since $\mathbb{P}=\int_{X}\mathbb{P}_{x}d\lambda\left(x\right)$,
it follows that $V_{0}$ in (\ref{eq:cp-v0a}) is isometric, i.e.,
\[
\left\Vert V_{0}g\right\Vert _{L^{2}\left(\mathbb{P}\right)}^{2}=\int_{X}\left|g\right|^{2}h\,d\lambda=\left\Vert g\right\Vert _{L^{2}\left(h\,d\lambda\right)}^{2},\quad\forall g\in C\left(X\right).
\]
To prove (\ref{eq:cp-v0b}), we must establish 
\begin{equation}
\int_{Sol_{\sigma}\left(X\right)}\left(g\circ Z_{0}\right)\psi\,d\mathbb{P}=\int_{X}g\left(x\right)\mathbb{E}\left(\psi\mid x\right)d\lambda\left(x\right).\label{eq:cp6}
\end{equation}
Since the space of the cylinder functions $\psi=\left(f_{0}\circ Z_{0}\right)\left(f_{1}\circ Z_{1}\right)\cdots\left(f_{n}\circ Z_{n}\right)$
is dense in $C\left(Sol_{\sigma}\left(X\right)\right)$, it suffices
to prove (\ref{eq:cp6}) for $\psi$. But then 
\[
\left(g\circ Z_{0}\right)\psi=\left(gf_{0}\right)\circ Z_{0}\left(f_{1}\circ Z_{1}\right)\cdots\left(f_{n}\circ Z_{n}\right),
\]
and so (\ref{eq:cp6}) follows from (\ref{eq:cp5-1}).\end{proof}
\begin{cor}
Fix $x\in X$, and set 
\begin{equation}
\mathbb{E}\left(\psi\mid x\right)=\int_{Z_{0}^{-1}\left(x\right)}\psi\,d\mathbb{P}_{x},\quad\psi\in L^{2}\left(\mathbb{P}\right).\label{eq:d1}
\end{equation}
For all $n\in\mathbb{N}$, if $A_{i}\subset X$, $i=1,\cdots,n$,
are Borel sets, then 
\begin{eqnarray}
 &  & \mathbb{P}_{x}\left(Z_{1}\in A_{1},Z_{2}\in A_{2},\cdots,Z_{n}\in A_{n}\right)\nonumber \\
 & = & \int_{A_{1}}\int_{A_{2}}\cdots\int_{A_{n}}h\left(y_{n}\right)P\left(dy_{n}\mid y_{n-1}\right)\cdots P\left(dy_{2}\mid y_{1}\right)P\left(dy_{1}\mid x\right).\label{eq:d2}
\end{eqnarray}
\end{cor}
\begin{proof}
Recall that $\chi_{A}\circ Z_{i}=\chi_{Z_{i}^{-1}\left(A\right)}$,
if $A\subset X$ is a Borel set; and 
\begin{equation}
Z_{i}^{-1}\left(A\right)=\left\{ x\in Sol_{\sigma}\left(X\right)\mid Z_{i}\left(x\right)\in A\right\} ,\label{eq:d3}
\end{equation}
where $Z_{i}\left(x_{0},x_{1},x_{2},\cdots\right)=x_{i}$ is the coordinate
mapping. Also, $P\left(\cdot\mid x\right)$ satisfies 
\begin{equation}
R\left(f\right)\left(x\right)=\int_{X}f\left(y\right)P\left(dy\mid x\right),\quad\forall x\in X.\label{eq:d4}
\end{equation}
Now set $f_{i}=\chi_{A_{i}}$, with $A_{i}\subset X$ Borel sets,
and apply the mapping 
\[
f_{i}\longrightarrow R\left(f_{1}R\left(f_{2}\cdots R\left(f_{n}h\right)\cdots\right)\right)\left(x\right).
\]

If we specialize (\ref{eq:d2}) to individual transition probabilities,
we get, $x\in X$, $A\subset X$ a Borel set, and 
\begin{align*}
\mathbb{P}\left(Z_{1}\in A\mid Z_{0}=x\right) & =\int_{A}h\left(y\right)P\left(dy\mid x\right);\\
\mathbb{P}\left(Z_{2}\in B,Z_{1}\in A\mid Z_{0}=x\right) & =\int_{A}\int_{B}h\left(y_{2}\right)P\left(dy_{2}\mid y_{1}\right)P\left(dy_{1}\mid x\right),\quad y_{1}\in A,y_{2}\in B.
\end{align*}

Note that, fix $n>1$, then 
\[
\mathbb{P}\left(Z_{n}\in A\mid Z_{0}=x\right)=\mathbb{P}_{x}\left(Z_{n}\in A\right)=R^{n}\left(\chi_{A}h\right)\left(x\right),\mbox{ and}
\]
\[
\mathbb{P}_{x}\left(Z_{n+1}\in B,Z_{n}\in A\right)=R^{n}\left(\chi_{A}R\left(\chi_{B}h\right)\right)\left(x\right)\neq\mathbb{P}_{x}\left(Z_{2}\in B,Z_{1}\in A\right),
\]
so it is not Markov. 
\end{proof}
Hence the transition from $n$ to $n+1$ gets more \textquotedblleft flat\textquotedblright{}
as $n$ increases, the transition probability evens out with time.

\subsection{Multi-Resolutions}

Let $X,\sigma,R,h$, and $\lambda$ be as in the setting of \thmref{sol}
above. In particular, we are assuming that:
\begin{enumerate}[label=(\roman{enumi}),ref=\roman{enumi},itemsep=0.5em]
\item $R\left(\left(f\circ\sigma\right)g\right)=fR\left(g\right)$, $\forall f,g\in C\left(X\right)$, 
\item $Rh=h$, $h\geq0$,
\item $\int R\left(f\right)d\lambda=\int_{X}fW\,d\lambda$, $\forall f\in C\left(X\right)$,
and
\item $\int_{X}h\left(x\right)d\lambda\left(x\right)=1$. 
\end{enumerate}
We then pass to the probability space $\left(Sol_{\sigma}\left(X\right),\mathbb{P}_{x},\mathbb{P}\right)$
from the conclusion in \thmref{sol}.
\begin{defn}
Let $\mathscr{H}$ be a Hilbert space, and $\left\{ \mathscr{H}_{n}\right\} _{n\in\mathbb{N}_{0}}$
a given system of closed subspaces such that $\mathscr{H}_{n}\subset\mathscr{H}_{n+1}$,
for all $n$. 

We further assume that $\cup_{n}\mathscr{H}_{n}$ is dense in $\mathscr{H}$,
and that a unitary operator $U$ in $\mathscr{H}$ satisfying $U\left(\mathscr{H}_{n}\right)\subset\mathscr{H}_{n-1}$,
for all $n\in\mathbb{N}$. Then we say that $\left(\left(\mathscr{H}_{n}\right)_{n\in\mathbb{N}_{0}},U\right)$
is a \emph{multi-resolution} for the Hilbert space $\mathscr{H}$.\end{defn}
\begin{thm}
\label{thm:mr}Let $\mathscr{H}=L^{2}\left(Sol_{\sigma}\left(X\right),\mathbb{P}\right)$
be the Hilbert space from the construction in \thmref{sol}, and let
$\mathscr{H}_{n}$ be the closed subspaces defined from the random
walk process $\left(Z_{n}\right)_{n\in\mathbb{N}}$. Finally, let
$U$ be the operator in part (\ref{enu:sol3}) of \thmref{sol}. Then
this constitutes a \uline{multi-resolution}\emph{.}\end{thm}
\begin{proof}
As indicated above, the setting is specified in \thmref{sol}, and
we set 
\[
\mathscr{H}:=L^{2}\left(Sol_{\sigma}\left(X\right),\mathbb{P}\right);
\]
and, for each $n\in\mathbb{N}$, let $\mathscr{H}_{n}\subset\mathscr{H}$,
be the closed subspace spanned by
\begin{equation}
\left\{ f\circ Z_{n}\mid f\in C\left(X\right)\right\} .\label{eq:mr1}
\end{equation}
Since 
\begin{equation}
f\circ Z_{n}=\left(f\circ\sigma\right)\circ Z_{n+1}\label{eq:mr2}
\end{equation}
it follows that $\mathscr{H}_{n}\subseteq\mathscr{H}_{n+1}$. It further
follows from \thmref{sol} that $\cup_{n\in\mathbb{N}}\mathscr{H}_{n}$
is dense in $\mathscr{H}$. And, finally, the unitary operator $U$
from part (\ref{enu:sol3}) of \thmref{sol} satisfies
\begin{equation}
U\left(\mathscr{H}_{n}\right)\subset\mathscr{H}_{n-1},\quad\forall n\in\mathbb{N}.\label{eq:mr3}
\end{equation}
\end{proof}
\begin{cor}
Let $X,\sigma,R,h,\lambda,\mathbb{P}$ be as stated above; and let
$\left(\left(\mathscr{H}_{n}\right),U\right)$ be the corresponding
multi-resolution from \thmref{mr}. 

Then $\mathscr{H}_{0}\simeq L^{2}\left(X,h\,d\lambda\right)$, and
$\cap_{n\geq0}U^{n}\mathscr{H}_{m}=\mathscr{H}_{0}$ holds for all
$m\in\mathbb{N}$. Finally, $U$ restricts to a unitary operator in
$\mathscr{H}\ominus\mathscr{H}_{0}$; and the spectrum of this restriction
is pure Lebesgue spectrum, i.e., there is a Hilbert space $\mathscr{K}$
(the multiplicity space) such that $U\big|_{\mathscr{H}\ominus\mathscr{H}_{0}}$
is unitarily equivalent to a subshift of the bilateral shift $S$
in $L^{2}\left(\mathbb{T},\mbox{Leb};\mathscr{K}\right)$, where $\mathbb{T}$
is the circle group $\left\{ z\in\mathbb{C}\mid\left|z\right|=1\right\} $,
and the bilateral shift is then given on functions $\psi\in L^{2}\left(\mathbb{T},\mbox{Leb};\mathscr{K}\right)$
by 
\[
\left(S\psi\right)\left(z\right)=z\psi\left(z\right),\;\psi:\mathbb{T}\longrightarrow\mathscr{K},\;z\in\mathbb{T},\;\mbox{multiplication by \ensuremath{z}.}
\]
\end{cor}
\begin{proof}
The conclusion follows from an application of the Stone-von Neumann
uniqueness theorem \cite{MR2042745} combined with the present theorems
in Sections \ref{sec:L1} and \ref{sec:EA} above. (For more details
on the spectral representation for operators with multi-resolution,
see also \cite{MR0220099}.)
\end{proof}

\section{\label{sec:Harm}Harmonic Functions from Functional Measures}

Let $R:C\left(X\right)\longrightarrow\mathcal{M}\left(X\right)$ be
as specified in (\ref{def:R}); and let the measure system $\left\{ P\left(\cdot\cdot\mid x\right)\right\} _{x\in X}$
be as specified in (\ref{def:cm}). 

Let $\left(\Omega,\mathscr{F}\right)$ be a measure space; i.e., $\mathscr{F}$
is a specified sigma-algebra of events in a given sample space $\Omega$,
and let $Z_{0}$ be an $X$-valued random variable, i.e., it is assumed
that $Z^{-1}\left(A\right)\in\mathscr{F}$ for every Borel set $A\subset X$.
For recent applications, we refer to \cite{MR2966130,MR2966144,MR2966145,MR3286496,MR3370362}.
\begin{thm}
\label{thm:harm}Let $R,X,\Omega,\mathscr{F}$, and $Z_{0}$ be as
specified above. Suppose $\left\{ \mathbb{P}_{x}\right\} _{x\in X}$
is a system of positive measures on $\Omega$ indexed by $X$, and
set 
\begin{equation}
h\left(x\right)=\mathbb{P}_{x}\left(Z_{0}^{-1}\left(x\right)\right),\quad x\in X.\label{eq:h1}
\end{equation}
Assume that 
\begin{equation}
\int_{X}\mathbb{P}_{y}\left(\cdot\cdot\right)P\left(dy\mid x\right)=\mathbb{P}_{x}\left(\cdot\cdot\right),\label{eq:h2}
\end{equation}
then $h$ in (\ref{eq:h1}) is harmonic for $R$, i.e., we have 
\begin{equation}
R\left(h\right)=h,\;\mbox{pointwise on \ensuremath{X}.}\label{eq:h3}
\end{equation}
\end{thm}
\begin{proof}
Using (\ref{eq:R}) in \defref{cm}, we get the following:
\begin{eqnarray*}
\left(Rh\right)\left(x\right) & = & \int_{X}h\left(y\right)P\left(dy\mid x\right)\\
 & \underset{\text{by \ensuremath{\left(\ref{eq:h1}\right)}}}{=} & \int_{X}\mathbb{P}_{y}\left(Z_{0}^{-1}\left(y\right)\right)P\left(dy\mid x\right)\\
 & \underset{\text{by \ensuremath{\left(\ref{eq:h2}\right)}}}{=} & \mathbb{P}_{x}\left(Z_{0}^{-1}\left(x\right)\right)=h\left(x\right),\quad x\in X.
\end{eqnarray*}
\end{proof}
\begin{cor}
Let $R$ be a transfer operator. Then if $\lambda\in\mathscr{L}_{1}\left(R\right)$,
then there is a solution $h\geq0$, on $X$, to $Rh=h$, and $\int_{X}h\left(x\right)d\lambda\left(x\right)=1$. \end{cor}
\begin{proof}
This is a conclusion of \corref{cp} and \thmref{harm}. Indeed, given
$\lambda\in\mathscr{L}_{1}\left(X\right)$, let $\left\{ \mathbb{P}_{x}\right\} _{x\in X}$
be the system from \thmref{sol}, then $h\left(x\right)=\mathbb{P}_{x}\left(Z_{0}^{-1}\left(x\right)\right)$
is the desired solution. \end{proof}
\begin{acknowledgement*}
The co-authors thank the following colleagues for helpful and enlightening
discussions: Professors Sergii Bezuglyi, Ilwoo Cho, Paul Muhly, Myung-Sin
Song, Wayne Polyzou, and members in the Math Physics seminar at The
University of Iowa.

\bibliographystyle{amsalpha}
\bibliography{ref}
\end{acknowledgement*}

\end{document}